\documentclass[a4paper,12pt,reqno]{amsart}
\usepackage{amssymb}
\usepackage{amsmath}
\usepackage{enumitem}
\usepackage{mathrsfs}
\usepackage[all]{xy}
\usepackage{tikz-cd}
\usepackage{mathtools}
\usepackage{hyperref}
\usepackage{url}
\usepackage{bbm}

\usepackage[OT2,T1]{fontenc}
\DeclareSymbolFont{cyrletters}{OT2}{wncyr}{m}{n}
\usepackage[british]{babel}

\usepackage[margin=1.3in]{geometry}
\numberwithin{equation}{section} \numberwithin{figure}{section}

\DeclareMathOperator{\Gal}{Gal} 
 
\DeclareMathOperator{\Spec}{Spec}
 
\DeclareMathOperator{\Hom}{Hom} 
\DeclareMathOperator{\im}{Im}

\DeclareMathOperator{\Br}{Br}

\let\H\undefined
\DeclareMathOperator{\H}{H}

\DeclareMathOperator{\wgcd}{wgcd}

\DeclareSymbolFont{cyrletters}{OT2}{wncyr}{m}{n}
\DeclareMathSymbol{\Sha}{\mathalpha}{cyrletters}{"58}
\DeclareMathSymbol{\Be}{\mathalpha}{cyrletters}{"42}

\renewcommand{\O}{\mathcal{O}}
\newcommand{\GL}{\textrm{GL}}

\newcommand{\Sp}{\mathrm{Sp}}
\newcommand{\GSp}{\mathrm{GSp}}

\newcommand{\x}{\mathbf{x}}
\newcommand{\y}{\mathbf{y}}

\renewcommand\P{\mathbb{P}}
\newcommand\A{\mathbb{A}}
\newcommand\Z{\mathbb{Z}}
\newcommand\N{\mathbb{N}}
\newcommand\Q{\mathbb{Q}}
\newcommand\R{\mathbb{R}}
\newcommand\C{\mathbb{C}}
\newcommand\GG{\mathbb{G}}

\newcommand\Gm{\GG_\mathrm{m}}

\newcommand{\sX}{\mathcal{X}}
\newcommand{\sY}{\mathcal{Y}}

\renewcommand{\a}{\mathbf{a}}

\newcommand{\sumstar}{\sideset{}{^*}\sum}

\newtheorem{lemma}{Lemma}

\newtheorem{theorem}[lemma]{Theorem}

\newtheorem{corollary}[lemma]{Corollary}

\theoremstyle{definition}
\newtheorem{example}[lemma]{Example}
\newtheorem{definition}[lemma]{Definition}
\newtheorem{remark}[lemma]{Remark}

\newtheorem{question}[lemma]{Question}

\numberwithin{lemma}{section}

\usepackage{color}

\newcommand{\dan}[1]{{\color{blue} \sf $\clubsuit\clubsuit\clubsuit$ Dan: [#1]}}
\newcommand{\nick}[1]{{\color{red} \sf $\diamondsuit\diamondsuit\diamondsuit$ Nick: [#1]}}
\newcommand{\steph}[1]{{\color{teal} \sf $\diamondsuit\diamondsuit\diamondsuit$ Steph: [#1]}}

\begin{document}

\title{Thin sets in weighted projective stacks}

\author{Stephanie Chan}
\address{
Department of Mathematics, University College London, Gower Street, London, WC1E 6BT, UK. 
}

\author{Daniel Loughran}
\address{
Department of Mathematical Sciences \\
University of Bath \\
Claverton Down \
Bath\\ 
BA2 7AY\\
UK.}


\author{Nick Rome}
\address{
Graz University of Technology, Institute of Analysis and
Number Theory, Steyrergasse 30/II, 8010 Graz, Austria.
}

\subjclass[2010]
{
(secondary)}

\begin{abstract}
	We prove an upper bound for the number of rational points of bounded height
	in a weighted projective stack which lie in a given thin subset. As a consequence, we show that  $100\%$ of hyperelliptic curves do not admit a prescribed non-trivial level structure.
\end{abstract}

\maketitle

\thispagestyle{empty}

\tableofcontents

\section{Introduction} \label{sec:intro}
\subsection{Thin sets}
Thin sets of rational points were introduced by Serre \cite[\S 3]{SerreGalois} in the study of the inverse Galois problem. Thin sets have now become fundamental in Diophantine geometry and appear throughout the study of rational points on varieties in various forms. A variety whose set of rational points is not thin is said to satisfy the \emph{Hilbert property}. It is well known that verifying the Hilbert property for the unirational variety $\A^n/G$ for a finite group $G$ would solve the Inverse Galois Problem for that group. A great deal of recent activity has been devoted to establishing this property for particular varieties, see for example \cite{HPs, HPls, HPds, HPdsw}, as well as the recent survey article \cite{FJ25}.
In another direction, it has recently become widely accepted that to achieve the asymptotic formula predicted by the Batyrev--Manin conjecture \cite{FMT89,BM90} for rational points of bounded height on Fano varieties, one needs to remove a thin subset. This perspective has also recently become essential to help explain counter-examples to Malle's conjecture on number fields of bounded discriminant, where one needs to consider thin sets of rational points on classifying stacks, see \cite[\S3.7]{LSBG} and \cite[\S9]{DY}.

Serre~\cite[Ch.~13, Thm.~3]{MW}, based on work of Cohen \cite{Coh81}, presented a new proof that the Hilbert property holds for projective space $\P^n$ by establishing, for $H$ the na{\"i}ve Weil height and for some $\gamma <1$, the bound
\begin{equation} \label{eqn:Serre}
\#\{x \in \P^n(\Q) : H(x) \leq B,\ x \in \Omega\} \ll B^{n+\frac{1}{2}} (\log B)^\gamma
\end{equation}
for any thin set $\Omega \subset \P^n(\Q)$.
From this result, one can deduce the well-known fact that 100\% of polynomials of degree $n$ have Galois group $S_n$. Moreover, Serre~\cite{Ser90} used this idea to bound the number of fibres in conic bundle surfaces which contain rational points. This idea was vastly generalised by Loughran--Smeets \cite{LS16} to produce a bound on the number of everywhere locally soluble fibres in more general fibrations over $\P^n$, a result which has spawned a vast and active research area (see \cite{LRS} for a summary of the history and \cite{genchat,manyconics} for some more recent results). Moreover, for suitable classes of Fano varieties, it was shown in \cite[Thm.~1.2]{BLsieving} that only $0\%$ of rational points lie in a given thin subset.

The advent of a stacky version of the Batyrev--Manin conjecture \cite{ESZB,DY} has brought on a flurry of interest in which ideas from the traditional study of rational points can be generalised to the stacky realm. Our aim in this paper is to study thin sets in a stacky setting. Our main result is a version of Serre's theorem for weighted projective stacks (see \S \ref{sec:thin_sets} for background).

\begin{theorem}\label{thm:stackyserre}
Let $\a = (a_0, \dots, a_n) \in \N^{n+1}$, let $k/\Q$ be a number field of degree $d$ and let $H_\a$ denote the weighted height function (see~Definition \ref{def:height}). For any thin subset $\Omega \subset \P(\a)(k)$ there exists $\gamma < 1$ such that
\[
\#\{ x \in \Omega : H_\a(x) \leq B \text{ and } x \in \Omega\}
\ll B^{a_0 + \dots + a_n - \frac{\min_i  a_i}{2}}  (\log B)^\gamma.
\]
\end{theorem}

\begin{remark}
The total number of rational points on $\P(\mathbf a)$ of height at most $B$ grows like $B^{a_0+\dots+a_n}$ \cite{Deng,BruinWPS, Darda}.  Hence we see that  $100\%$ of points in weighted projective space lie outside any given thin set.
\end{remark}

We prove Theorem \ref{thm:stackyserre} following Serre's method \cite[Ch.~13, Thm.~3]{MW}. To achieve this we develop a lopsided version of the large sieve over number fields; this is Theorem \ref{thm:lopsided}. We expect this general large sieve to be applicable to other counting problems on weighted projective stacks and we developed it for this larger purpose. Theorem \ref{thm:stackyserre} is a sample application of our sieve, but it may be possible that stronger bounds for Theorem \ref{thm:stackyserre} are obtainable from the determinant method.

From a more qualitative perspective, we also introduce the natural analogue of the Hilbert property for stacks in \S\ref{sec:thin_sets} and show that it can sometimes be deduced from the scheme version of the problem.
\begin{theorem}\label{thm:HP}
Let $\mathcal{X}$ be an algebraic stack over a field $k$. Let $X$ be a variety which satisfies the Hilbert property and assume that there is a dominant morphism $X \to \mathcal{X}$ with geometrically integral generic fibre. Then $\mathcal{X}$ satisfies the Hilbert property.
\end{theorem}

Theorem \ref{thm:stackyserre} gives an analytic proof of the Hilbert property for weighted projective stacks over number fields, in the spirit of Serre's original proof for projective space. Applying Theorem \ref{thm:HP} to the quotient map $\mathbb{A}^{n+1} \setminus \{0\} \to \P(\a)$ also shows that weighted projective stacks satisfy the Hilbert property; naturally this uses Serre's result that $\mathbb{A}^{n+1}$ satisfies the Hilbert property. However Theorem~\ref{thm:stackyserre} also allows one to deduce the stronger claim that $\P(\a)$ satisfies the integral Hilbert property (see Corollary \ref{cor:IHP}).

\subsection{Serre's question for weighted projective space}
In the light of \eqref{eqn:Serre}, Serre asked whether one might be able to improve the exponent to save $B$, rather than just $B^{\frac{1}{2}}$. In the case of a type I thin set, this problem is now known as the Dimension Growth conjecture which predicts that given an integral projective variety $X$ then the number of points of absolute Weil height at most $B$ on $X$ should grow no faster than $O_{X,\epsilon}(B^{\dim X + \epsilon})$. This conjecture was first formulated for hypersurfaces by Heath-Brown \cite{cubicsin10} before Serre extended it to all projective varieties. The conjecture was first established for geometrically integral quadratic hypersurfaces by Heath-Brown~\cite{HBannals} then for varieties of degree at least 6 by Browning--Heath-Brown--Salberger \cite{BHBS} and ultimately for all integral projective varieties by Salberger \cite{SalbergerPLMS}. Moreover, Salberger's bounds only depend on the dimension and degree of $X$ thus confirming a uniform version of the conjecture first stated by Heath-Brown~\cite{HBannals}
. This uniformity has since been extended by Castryck--Cluckers--Dittmann--Nguyen~\cite{CCDN} who were able to turn the $B^\epsilon$ into a log power and prove explicit polynomial dependence of the implicit constant on the degree (provided the degree is at least 5). This final result has been extended to number fields by Paredes--Sasyk~\cite{PS}.

This question has also received intensive consideration in the setting of Type II thin sets. Broberg~\cite{Broberg} established the conjecture when the dimension of projetive space is 2 or 3 and the thin set has degree at least 3. Heath-Brown--Pierce~\cite{HBP} established the conjecture in the case of cyclic covers for dimension at least 10, and there have been a number of bounds in special situations improving on Serre's original result \cite{Munshi,Bonolis,BP,BPW}. The conjectured upper bound for Type II thin sets of degree at least 4 was established recently by Buggenhout--Castryck--Cluckers--Santens--Vermeulen~\cite{BCCSV}.

With all this activity in mind, we feel it natural to ask the same question in the setting of weighted projective space.




\begin{question} \label{que:upper_bound}
Let $k$ be a number field, $\a \in \N^{n+1}$ a weight vector and $\Omega \subset \P(\a)(k)$ a thin subset. Then is
\[
\#\{ x \in \Omega : H_\a(x) \leq B )\} \ll_{f, \epsilon}
 B^{a_0 + \dots + a_n - \min_i  a_i+ \epsilon} ?
\]
\end{question}

%

There already exist a few examples of Type I thin sets of weighted projective spaces where even stronger bounds have been established.
Bonolis--Browning~\cite{dp1} have established a bound of size $O_n(B^{3-1/20}\log^2 B)$ for a certain family of hypersurfaces inside $\P_\Q(1,1,2,n)$. Moreover, Glas--Hochfilzer~\cite{GH} have shown the bounds $O(B^{\frac{3}{2} +\epsilon})$ and $O(B^{2+\epsilon})$ for del Pezzo surfaces of degree 2 and 1, respectively, over any global field of characteristic distinct from 2 or 3. These naturally arise as hypersurfaces inside $\P(1,1,1,2)$ and $\P(1,1,2,3)$, respectively. We note that the famous conjecture of Manin would predict $O(B^{1+\epsilon})$ points for both, away from the lines.



\subsection{Application to level structures on hyperelliptic curves} \label{sec:intro_hyperelliptic}
Stacks naturally arise as moduli spaces. Thin sets on these moduli spaces often correspond to objects with special arithmetic properties. For weighted projective stacks, Theorem~\ref{thm:stackyserre} then shows that $100\%$ of the objects parametrised by these moduli spaces do not possess these properties. 

The stack $\P(4,6)$ is a natural compactification of the moduli stack $\mathcal{M}_{1,1}$ of elliptic curves. More generally, the moduli stack of \emph{odd} hyperelliptic curves, by which we mean a hyperelliptic curve with a marked Weierstrass point, is isomorphic to an open subset of the weighted projective stack $\P(4,6,\dots, 4g + 2)$. A point $(a_4,a_6,\dots,a_{4g+2})$ corresponds to the curve
$$ y^2 = x^{2g + 1} + a_4x^{2g-1} + a_6x^{2g -2} + \dots + a_{4g + 2};$$
see \cite[Prop.~5.9]{HP23} or \cite[Prop.~4.2(1)]{Fed14}. Natural covers of this moduli stack arise via adding \textit{level structure} by imposing restrictions on the Galois representation. More specifically, let $A$ be a principally polarised abelian variety of dimension $g$ over a number field $k$ and $n \in \N$. Let $\GSp_{2g}$ denote the group of symplectic similitudes and $\nu: \GSp_{2g} \to \Gm$ the multiplier homomorphism. The absolute Galois group $\Gamma_k$ of $k$ acts on the $n$-torsion points $A[n]$ of $A$ over $\bar{k}$. Choosing a basis yields an isomorphism $A[n] \cong (\Z/n\Z)^{2g}$ and thus a representation $\rho_{A,n}: \Gamma_k \to \GL_{2g}(\Z/n\Z)$. The Galois equivariance of the Weil pairing, which defines a non-degenerate skew-symmetric bilinear form on $A[n]$, implies that the image of $\rho_{A,n}$ lies in $\GSp_{2g}(\Z/n\Z)$. Different choices of basis differ by conjugation, thus the image of $\rho_{A,n}$ is well-defined up to conjugation. 
For a subgroup $G \subseteq \GSp_{2g}(\Z/n\Z)$, we say that $A$ \textit{admits a level-$G$ structure} if the image of $\rho_{A,n}$ is contained in a conjugate of $G$. The existence of roots of unity in $k$ imposes additional conditions on $\im \rho_{A,n}$, namely that the multiplier of $\im \rho_{A,n}$ equals the image of the cyclotomic character; to avoid such technicalities we only consider subgroups $G$ with $\nu(G) = (\Z/n\Z)^\times$. The image of the Galois representation will always have this property when $k = \Q$, for example.  We say that a smooth projective curve \textit{admits a level-$G$ structure} if its Jacobian does.

\begin{theorem} \label{thm:level_structures}
	Let $k$ be a number field, let $g \in \N$, let $n \in \N$ be odd, and let $G \subseteq \GSp_{2g}(\Z/n\Z)$ be a proper subgroup such that $\nu(G) = (\Z/n\Z)^\times$. Then 100\% of odd hyperelliptic curves over $k$ admit no level $G$-structure when ordered by weighted projective height.
\end{theorem}

We compare this with the more well-studied case of elliptic curves. Here $\GSp_{2} = \GL_2$ and the multiplier $\nu$ is simply the determinant. A great deal of work has been done establishing the correct order of magnitude for the number of elliptic curves with a specified level $G$-structure when ordered by the weighted projective height. We make no attempt to summarise this vast field here but instead direct the reader to the recent paper of Phillips~\cite{tristan} who, working over an arbitrary number field, provides asymptotics for a wide range of $G$ and also provides a detailed history of the problem. While results of this type offer more precise bounds than Theorem \ref{thm:level_structures}, the advantage of our result is that it applies uniformly to all possible odd levels structures and odd hyperelliptic curves of any genus over any number field.

There do exist general results in the spirit of Theorem \ref{thm:level_structures} in the literature. Zywina \cite{Zyw10}  showed that, taking into account the restriction imposed by the Weil pairing, $100\%$ of elliptic curves over number fields have largest possible Galois action. His result  was subsequently generalised to families of abelian varieties \cite[Thm.~1.1]{LSTY19}, with a slightly weaker result over $\Q$, and a more explicit version for families of hyperelliptic curves \cite[Thm.~1.2]{LSTY20}. The major distinction between these works and our result is that in all previous papers hyperelliptic curves were parametrised by points in affine space $\mathbb{A}^{2g}, \mathbb{A}^{2g+1}$ or $\mathbb{A}^{2g+2}$. Our approach is the first to provide bounds for the number of curves with given level $G$-structure when ordered by weighted projective space, taking the true stacky nature of the moduli space into account.



\subsection{Conventions}
Let $\mathcal{X}$ be an algebraic stack over a field $k$. The $k$-rational points on $\mathcal{X}$ form a groupoid, and when counting one should use the groupoid cardinality, i.e.~weigh each isomorphism class by the reciprocal of its automorphism group. Since we are only interested in upper bounds these subtleties are not important to us. So  to keep notation light we denote by $\mathcal{X}(k)$ the \textit{set} of isomorphisms classes of objects in the groupoid of $k$-points on $\mathcal{X}$. 

\subsection{Structure of the paper} In \S\ref{sec:thin_sets} we recall various background on stacks and thin sets. We also prove Theorem \ref{thm:HP}. In \S \ref{sec:large_sieve} we prove our main technical result (Theorem \ref{thm:lopsided}:, a lopsided version of the large sieve over number fields. We expect this result to find other applications. In \S \ref{sec:WPS} we prove Theorem \ref{thm:stackyserre}, as well as a version of the integral Hilbert property for weighted projective stacks (Corollary \ref{cor:IHP}) . Finally in \S \ref{sec:hyperelliptic} we prove Theorem \ref{thm:level_structures} using Theorem \ref{thm:stackyserre} and a framework for moduli stacks of abelian varieties with level structure.

\subsection*{Acknowledgements}
This project began as a working group at the WAARP conference held in Bristol in September 2023. The authors are very grateful to J. Lyczak and R. Paterson for organising the event, and the funding bodies who made it possible. The authors thank Sebastian Monnet for his enthusiasm for stacks and contributions at the inception of this project. Thanks also to Tristan Phillips for discussions on level structures and Tim Santens for comments on Question \ref{que:upper_bound}. Daniel Loughran was supported by UKRI Future Leaders Fellowship \texttt{MR/V021362/1}. Nick Rome was funded by  FWF project ESP 441-NBL. 


\section{Stacks and thin sets} \label{sec:thin_sets}
\subsection{Weighted projective stacks}
\begin{definition}
Let $\a = (a_0,\dots,a_n) \in \N^{n+1}$; we call $\a$ a \emph{weight vector}. 
The weighted projective stack $\P(\a)$ is defined as the quotient stack
\begin{equation} \label{def:WPS}
\A^{n+1}\setminus\{0\} \to \P(\a) := [\A^{n+1}\setminus\{0\} / \Gm]
\end{equation}
where $\Gm$ acts via $\lambda\cdot (x_0,\dots,x_n) = (\lambda^{a_0} x_0,\dots, \lambda^{a_n}x_n).$
This is a proper smooth toric Deligne--Mumford stack over $\Z$.
\end{definition}

For $\a = (1,\dots, 1)$, one recovers the usual definition of the projective space $\P^n$. However, for different weight vectors the resulting weighted projective stack will not be a scheme. For any field $k$ we have 
$$\P(\a)(k) = \{ (x_0, \dots, x_n) \in k^{n+1}\setminus\{0\} \}/\sim.$$
where $\sim$ denotes the equivalence relation
$$(x_0,\dots,x_n) \sim (\lambda^{a_0}x_0,\dots, \lambda^{a_n}x_n), \text{ for all } \lambda \in k^\times.$$
This follows from Hilbert Theorem 90 and the fact that the quotient map \eqref{def:WPS} is a $\Gm$-torsor.

The notion of height on a stack has been introduced in quite a general setting by Ellenberg--Satriano--Zureick-Brown~\cite{ESZB} and by Darda--Yasuda~\cite{DY}. For weighted projective spaces, Deng~\cite{Deng} first wrote down a height, and this notion coincides with these more general stacky notions of height \cite[\S3.3]{ESZB}.

\begin{definition}\label{def:height}
Let $k$ be a number field and let $\mathbf x =(x_0,\dots, x_n) \in k^{n+1}$. We define the \emph{weighted gcd}\ of $\mathbf x$ to be 
\[
\wgcd(\mathbf x) := \prod_{\mathfrak{p}} \mathfrak{p}^{\min_i \left\lfloor \frac{v_{\mathfrak{p}}(x_i)}{a_i}\right\rfloor}.
\]
Let $\mathbf x = (x_0,\dots, x_n) \in \P(\mathbf a)(k)$. Then the height of $\x$ is
\[
H_\a(\x) := \frac{1}{N_{k/\Q}(\wgcd(\x))} \prod_{v \mid \infty_k}\max_i \vert x_i \vert_v^{\frac{1}{a_i}}.
\]
\end{definition} 

As with rational points on projective space, we may pick a representative of any point $x \in \P(\a)(k)$ of the form $\x = (x_0, \dots, x_n) \in \mathcal{O}_k^{n+1}$. Moreover we may assume that this representative is chosen so that the ideal $\wgcd(\x)$ is equal to a fixed representative of some element of the class group of $k$.


\subsection{Hilbert property for stacks}\label{ss:HP}
Let $V$ be a variety over a field $k$. Recall that a thin subset of $V(k)$ is defined to be any subset contained in a finite union of subsets which are either contained in a proper closed
subvariety of $V$ (Type I), or contained in $\pi(W(k))$, where $W \xrightarrow{\pi} V$ is a generically
finite dominant morphism of degree exceeding 1 with $W$ an integral variety over $k$ (Type II). This notion has been extended by Darda--Yasuda \cite[Def.~5.4]{DY} to stacks.

\begin{definition} \label{def:thin}
A morphism of finitely presented integral stacks $f:\mathcal{Y}\rightarrow\mathcal{X}$ is \textit{birational} if there exist open dense substacks $\mathcal{V}\subset\mathcal{Y}$ and $\mathcal{U}\subset\mathcal{X}$ such that $f$ restricts to an isomorphism $\mathcal{V}\cong \mathcal{U}$. 

We say that $f$ is \textit{thin} if it is non-birational, representable, and generically finite onto the image. A thin subset of $\mathcal{X}(k)$ is a subset contained in $\cup_{i=1}^n f_i(\mathcal{Y}_i(k))$ for thin morphisms $f_i:\mathcal{Y}_i\rightarrow\mathcal{X}$ and some $n$. 
\end{definition}

This may be interpreted in the familiar language of Type I and II thin sets. 
As $f$ is representable, its degree is defined as follows (see \cite[Def.~1.15]{Vist}). Take a dominant morphism $Z \to \mathcal{X}$ where $Z$ is a scheme. Then we define $\deg(\sY \to \sX) = \deg(\sY \times_{\sX} Z \to Z)$, where the latter is a finitely presented generically finite morphism of schemes and the degree is defined in the usual may. This definition is independent of the choice of $Z$ \cite[Lem.~1.16]{Vist}.  Then we say that a subset has \emph{Type I} if it is contained in a Zariski-closed substack and \emph{Type II} if it has the form $\pi(\sY(k))$, where $\sY$ is integral and $\pi:\sY \to \sX$ is a finitely presented generically finite dominant morphism of degree at least $1$. Note that to prove that a set is thin, up to Type I thin sets,  in the definition of Type II thin set it suffices to consider the case where $\sY \to \sX$ finite (rather than just generically finite). Moreover since closed immersions are finite, in all we may always assume that $\sY \to \sX$ is finite.

\begin{definition}
A finitely presented integral stack $\sX$ satisfies the \emph{Hilbert property} over $k$ if $\sX(k)$ is not thin.
\end{definition}

If $\sX$ is birational to a scheme $X$, then $\sX$ satisfies the Hilbert property if and only if $X$ satisfies the Hilbert property. In particular this definition only gives genuinely new phenomenon for stacks with non-trivial generic stabiliser.

\begin{example}
	Let $G$ be a finite group and consider the classifying stack $BG$. Then it is not too difficult
	to see that the Hilbert property holds for $BG[k]$ if and only if $G$ is a Galois
	group over $k$, i.e.~the inverse Galois problem having a solution for $G$ over $k$. This follows
	from the description of thin sets on $BG$ from \cite[\S2.2]{LSBG}.
\end{example}

With this definition in hand, we are ready to establish Theorem \ref{thm:HP}.

\begin{proof}[Proof of Theorem \ref{thm:HP}]
Assume for a contradiction that $\mathcal{X}(k)$ is thin. Replacing $\mathcal{X}$ by an open dense substack if necessary, there exists non-birational finite morphisms $f_i: \mathcal{Y}_i \to \mathcal{X}$ for $i=1,\dots, n$ such that $\mathcal{X}(k) = \cup_{i = 1}^n  f_i(\mathcal{Y}(k))$. Let $\pi: X \to \mathcal{X}$ be as in Theorem \ref{thm:HP} and for $i = 1,\dots,n$ consider the diagram 
$$
\xymatrix{
Y_i:= \mathcal{Y}_i \times_\mathcal{X} X \ar[r]^{\quad \quad g_i} \ar[d] & X \ar[d]^{\pi} \\
\mathcal{Y}_i \ar[r]^{f_i} & \mathcal{X}.
}
 $$
Note that $Y_i$ is a scheme since $X$ is a scheme and $f_i$ was assumed to be finite, hence representable. The universal property of fibre products implies that $X(k) = \cup_{i = 1}^n  g_i(Y(k))$. As $X$ satisfies the Hilbert property, to obtain a contradiction it suffices to show that no $g_i$ admits a rational section. However the generic fibre of $f_i$ is a finite scheme, and a finite scheme cannot obtain a rational point after passing to a field extension which admits no non-trivial finite subfields, as is the case here since the generic fibre of $X \to \mathcal{X}$ is geometrically integral.
\end{proof}

\section{A lopsided large sieve over number fields} \label{sec:large_sieve}
Let $\a \in \N^{n+1}$ be a weight vector. We assume without loss of generality that $a_0 \leq a_1 \leq \dots \leq a_n$. Let $k$ be a number field of degree $d$ and fix an integral basis $\theta_1, \ldots, \theta_d$ for $\mathcal{O}_k$.

Fix a system of representatives for the ideal classes in $\mathcal{O}_k$. We may choose a representative for each element of $k^{n+1} \setminus \{0\}$ with respect to action of $k^\times$, which lies in $\mathcal{O}_k^{n+1}$ such that the weighted gcd is one of these fixed representatives. This determines the representative only up to the action of the units, for which we must still account.
Let 
\[
H_{\a, \infty}(\x) = \prod_{v \mid \infty} \max_i\{\vert x_i\vert_v^\frac{1}{a_i}\}.
\]
The weighted height $H_\a(x)$ introduced in Definition \ref{def:height} is then $\frac{H_{\a, \infty}(\x)}{N_{k/\Q}(\wgcd_\a(\x))}$. 
After having chosen our representatives, we see that the weighted gcd takes one of a finite number of values and thus the infinite part of the height will play the pivotal r{\^o}le. 
The aim of this section is to establish an appropriate large sieve for tuples in $\mathcal{O}_k$, modulo the action of the units via the weighted action, for which this  infinite part of the height is bounded.

\begin{theorem}\label{thm:lopsided}
    Let $Q \geq 1$, let $m \in \N$ and let $(\Omega_{\mathfrak{p}^m})_{\mathfrak{p}}$ be a collection of sets of residue classes in $\mathcal{O}_k^{n+1}/\mathfrak{p}^m$. Let $\nu_{\mathfrak{p}^m}$ denote the uniform measure on $\mathcal{O}_k^{n+1}/\mathfrak{p}^m$ and $\mu_k(\cdot)$ the ideal M{\"o}bius function in $k$. Then we have
    \[
\#\{  \x \in \mathcal{O}_k^{n+1}/ \mathcal{O}_k^\times : H_{\a,\infty}(\x) \ll_k B \text{ and } \x \bmod \mathfrak{p} \not \in \Omega_{\mathfrak{p}^m} \text{ for all } N_{k/\Q}(\mathfrak{p}) \leq Q\}
    \]
    \[
        \ll_{k,\a} \frac{\prod_{i=0}^n(B^{a_i} + Q^{2m})}{G(Q)}
    \]
    where
    \[
    G(Q) = \sum_{N_{k/\Q}(\mathfrak{q}) \leq Q} \mu_k(\mathfrak{q})^2
    \prod_{\mathfrak{p}\mid \mathfrak{q}}\frac{\nu_{\mathfrak{p}^m}(\Omega_\mathfrak{p}^m)}{1-\nu_{\mathfrak{p}^m}(\Omega_\mathfrak{p}^m)}.
    \]
\end{theorem}

\begin{remark}
While there do exist a number of generalisations of the traditional large sieve to number fields, most notably the work of Huxley~\cite{Huxley} and Schaal~\cite{Schaal}, we were unable to find a suitable multidimensional number field analogue in the existing literature. We take the opportunity to also sieve modulo prime powers as we anticipate this having future application to the study of soluble of fibres in families of varieties in the spirit of the work of Serre~\cite{Ser90} or Loughran--Smeets~\cite{LS16}.
\end{remark}


If the unit group of $\O_k$ is finite (e.g.~$k = \Q$) then it is not so difficult to prove this result from existing sieves in the literature. However greater care must be taken when there is an infinite unit  group, particularly constructing a fundamental domain for the action of the units. This is our first task, which we do in a similar manner to Deng~\cite[Prop.~4.2]{Deng} and Bruin--Manterola Ayala~\cite[Prop.~3.6]{BruinWPS}.
Let $\ell : k^\times \rightarrow \R^{r+s} $ be the standard logarithmic embedding, where $r$ denotes the number of real embeddings of $k$ and $2s$ the number of complex embeddings. 
Let $(\mathbf{u}_1, \dots, \mathbf{u}_{r+s-1})$ be a basis for the image of the unit group $\ell(\mathcal{O}_k^\times)$ in $\R^{r+s}$.
Define
\[
F\coloneqq [0,1)\mathbf{u}_1 + \dots + [0,1)\mathbf{u}_{r+s-1} \subseteq \R^{r+s},
\] where $d_i$ is the local degree of the $i^{\mathrm{th}}$ embedding.
 For any $T \geq 1$, denote
\[
F(T) \coloneqq F + (d_1, \dots, d_{r+s}) (-\infty, \log T].
\]
Then we write $S_{F,\a}(T)$  for the set of $(\mathbf{z}_1, \dots, \mathbf{z}_{r+s})\in (\R^{r}\times \C^s)^{(n+1)}$ such that
\[
\left(\max_{i\in\{0,\dots, n\}} \vert z_{1,i}\vert^{\frac{1}{a_i}}, \dots, \max_{i\in\{0,\dots, n\}} \vert z_{r+s,i}\vert^{\frac{1}{a_i}} \right)
\in \exp(F(T)).
\] 


\begin{lemma}
The set $S_{F,\a}(T^{\frac{1}{d}})$ is a fundamental domain for the free part of the unit group for points of height at most $T$ on $\P(\a)$.
\end{lemma}
\begin{proof}

Let $\nu_1,\dots,\nu_{r}$ be all the real places and $\nu_{r+1},\dots,\nu_{r+s}$ be all the complex places.
For any $j\in\{1,\dots,r+s\}$, let $\sigma_j$ denote the embedding associated to $\nu_j$.
Let $\x \in \mathcal{O}_k^{n+1}$. Then for any unit $\epsilon$, we have 
\[
\max_{i\in\{0,\dots, n\}} \vert \sigma_j(\epsilon^{a_i} x_i)\vert^{\frac{1}{a_i}}
=
\max_{i\in\{0,\dots, n\}}\vert \sigma_j(\epsilon)\vert
\max_{i\in\{0,\dots, n\}} \vert \sigma_j( x_i)\vert^{\frac{1}{a_i}}
=
\max_{i\in\{0,\dots, n\}} \vert \sigma_j( x_i)\vert^{\frac{1}{a_i}}.
\]
Therefore for any point in weighted projective space there are at most $w_k$ (the number of roots of unity in $k$)
representatives with
\[
\left(\max_{i\in\{0,\dots, n\}} \vert \sigma_{1}(x_i)\vert^{\frac{d_1}{a_i}}, \dots, \max_{i\in\{0,\dots, n\}} \vert \sigma_{r+s}(x_i)\vert^{\frac{d_{r+s}}{a_i}}\right)
\in \exp(F(\infty)).
\] 
Furthermore for any $T$, 
\[
\exp(F(T)) = \left\{(X_1, \dots, X_{r+s}) \in \exp(F(\infty)): X_1 \cdots X_{r+s} \leq T^d\right\}.
\]
Therefore we conclude that for any of the $w_k$ representatives, we have 
\begin{equation*}
H_{\a,\infty}(\x) \leq T \iff (\sigma_{1} \x, \dots, \sigma_{r+s} \x) \in S_{F,\a}(T^{\frac{1}{d}}). \qedhere
\end{equation*}
\end{proof}

To prove Theorem \ref{thm:lopsided}, we will appeal to the general framework developed by Kowalski \cite{Kow} for the large sieve. In particular, we take as our sieve setting $Y = \mathcal{O}_k^{n+1}$, $\Lambda = \{\mathfrak{p} \subset \mathcal{O}_k\}$ and $\rho_{\mathfrak{p}}: \mathcal{O}_k^{n+1} \rightarrow \left( \mathcal{O}_k/\mathfrak{p}^m \right)^{n+1}$ the standard reduction map. For any $T \geq 1$, we take as our sifting set $X =\mathcal{O}_k^{n+1} \cap S_{F, \a}(T^{\frac{1}{d}})$  and as our prime sieve support $\mathcal{L}^* = \{ \mathfrak{p} \subset \mathcal{O}_k: N_{k/\Q}(\mathfrak{p}) \leq Q\}$. 
Thus applying \cite[Prop.~2.3]{Kow}, we have
\[\#\left\{ \x \in \mathcal{O}_k^{n+1}\cap S_{F,\a}(T^{\frac{1}{d}}):\x \bmod \mathfrak{p}^m \not \in \Omega_{\mathfrak{p}^m}  \text{ for all }N_{k/\Q}(\mathfrak{p}) \leq Q\right\}\\
\leq \Delta/G(Q),\]
where $\Delta$ is the large sieve constant which we must now determine. To ease notation, we suppress the subscript of $N_{k/\Q}$.
The large sieve constant in this setting is defined to be the smallest non-negative real number $\Delta$ such that the inequality
\begin{equation}\label{eq:deltadef}
\sum_{\mathfrak{q} \leq Q} \mu_k^2(\mathfrak{q}) \sumstar_{\chi \bmod{\mathfrak{q}^m}} \left\vert \sum_{ \x \in \mathcal{O}_k^{n+1}\cap S_{F,\a}(T^{\frac{1}{d}}) } c(\x) \chi(\x) \right \vert^2
\leq \Delta \sum_{ \x \in \mathcal{O}_k^{n+1}\cap S_{F,\a}(T^{\frac{1}{d}}) } \vert c(\x)\vert^2
\end{equation}
holds for any square summable sequence $c(\x)$. Here, the $\chi$ sum runs over all primtive additive characters of $\left( \mathcal{O}_k/\mathfrak{q}^m\right)^{n+1}$.

\begin{lemma}\label{lem:chars}
Any additive character $\chi$ of $\left( \mathcal{O}_k/\mathfrak{q}^m\right)^{n+1}$ is expressible as
\[
\chi(\x) = e\left( \frac{\mathbf{m}^{(1)} \cdot (x_{1,1}, \ldots, x_{1,d}) + \ldots + \mathbf{m}^{(n+1)} \cdot (x_{n+1,1}, \ldots, x_{n+1,d})}{N(\mathfrak{q}^m)}\right),
\] where $\mathbf{m}^{(i)} \in \Z^d$ for each $i$ and $\x = (x_1, \ldots, x_{n+1})\in \mathcal{O}_k^{n+1}$ with $x_i = \sum_{j=1}^d x_{i,j} \theta_j$. Moreover, for every $i$ and $j$, one has 
\begin{equation}\label{eq:msmall}
\vert m^{(i)}_j \vert \leq \frac{1}{2} N(\mathfrak{q}^m),
\end{equation} and if $\chi$ is non-principal then 
\begin{equation}\label{eq:mbig}
\max_{i,j} \vert m_j^{(i)} \vert \gg_k N(\mathfrak{q}^m)^{1-1/d}.
\end{equation} 
\end{lemma}

\begin{proof}
When $n=0$, this is precisely the statement of \cite[Lemma 1]{Huxley}. We deduce the higher dimensional version by recalling that the dual of a product of finite groups is the product of the duals.
\end{proof}


With this in hand we may finish the estimation of the large sieve constant. 

\begin{proof}[Proof of Theorem \ref{thm:lopsided}]
Throughout this proof we write vectors in $(\R^d)^{n+1}$ as vectors in $\R^{d(n+1)}$ using the lexicographic ordering.
Let $U$ denote the image of the region $S_{F,\a}(T^{\frac{1}{d}})$ under the embedding $k^{n+1} \hookrightarrow\R^{d(n+1)}$. Let $\mathcal{R} \subset\R^{d(n+1)}$ be the smallest product of intervals containing $U$, and denote by $R_j$ the length of the projection onto the $j^\text{th}$ coordinate axis. 
Given $\y \in \Z^{d(n+1)}$, let $\widetilde{\y}$ denote the vector $(y_1 \theta_1 + \ldots + y_d\theta_d, \ldots,y_{nd+1} \theta_1 + \ldots + y_{d(n+1)}\theta_d)$ in $\mathcal{O}_k^{n+1}$.
Then, we may rewrite
\begin{multline*}
\sum_{\mathfrak{q} \leq Q} \mu_k^2(\mathfrak{q}) \sumstar_{\chi \bmod{\mathfrak{q}^m}} \left\vert \sum_{ \x \in \mathcal{O}_k^{n+1}\cap S_{F,\a}(T^{\frac{1}{d}}) } c(\x) \chi(\x) \right \vert^2\\
=
\sum_{\mathfrak{q} \leq Q} \mu_k^2(\mathfrak{q}) \sumstar_{\chi \bmod{\mathfrak{q}^m}} \left\vert \sum_{ \y \in \Z^{d(n+1)} \cap \mathcal{R} } \widetilde{c}(\y) \chi(\widetilde{\y}) \right \vert^2,
\end{multline*} where $\widetilde{c}(\y) = c(\widetilde{\y})$ if $\widetilde{y} \in S_{F, \mathbf{a}}(T^{\frac{1}{d}})$ and 0 otherwise.

By Lemma \ref{lem:chars}, there are a collection of at most $O(Q^{md(n+1)})$ vectors $\boldsymbol{\underline{\alpha}}=(\boldsymbol{\alpha}_1, \ldots, \boldsymbol{\alpha}_{n+1})\in\Q^{d(n+1)}$ where for each $\boldsymbol{\underline{\alpha}}$ there exists a squarefree integral ideal $\mathfrak{q}$ of norm at most $Q$ such that $\alpha_i = \frac{\mathbf{m}^{(i)}}{N(\mathfrak{q})^m}$ for some $\mathbf{m}^{(i)} \in \Z^d$ satisfying \eqref{eq:msmall} and \eqref{eq:mbig}. 
Thus we may write 
\[
\sum_{\mathfrak{q} \leq Q} \mu_k^2(\mathfrak{q}) \sumstar_{\chi \bmod{\mathfrak{q}^m}} \left\vert \sum_{ \y \in \Z^{d(n+1)} \cap \mathcal{R} } \widetilde{c}(\y) \chi(\widetilde{\y}) \right \vert^2
=
\sum_{i=1}^{cQ^{md(n+1)}}
 \left\vert \sum_{ \y \in \Z^{d(n+1)} \cap \mathcal{R} } \widetilde{c}(\y) e\left(\boldsymbol{\alpha}^{(i)} \cdot \y \right) \right \vert^2.
\]
Moreover, for distinct $\boldsymbol{\alpha}$ and $\boldsymbol{\alpha}'$, we have that
\[
\max_{i,j} \left \vert \frac{m^{(i)}_j}{N(\mathfrak{q})^m} - \frac{m^{'(i)}_j}{N(\mathfrak{q'})^m} \right \vert 
\gg_k \frac{1}{N(\mathfrak{q})^{m/d}N(\mathfrak{q'})^{m/d}}.
\]
Hence we may apply Huxley's large sieve \cite[Theorem 1]{Huxley} with $\delta_j = Q^{-\frac{2m}{d}}$ for each $j=1, \ldots, d(n+1)$ and $N_j = R_j$ which yields
\begin{multline*}
\sum_{\mathfrak{q} \leq Q} \mu_k^2(\mathfrak{q}) \sumstar_{\chi \bmod{\mathfrak{q}^m}} \left\vert \sum_{ \x \in \mathcal{O}_k^{n+1}\cap S_{F,\a}(T^{\frac{1}{d}}) } c(\x) \chi(\x) \right \vert^2\\
\leq 
\prod_{j=1}^{d(n+1)} (R_j^{\frac{1}{2}} + Q^{\frac{m}{d}})^2 
 \sum_{ \y \in \Z^{d(n+1)} \cap \mathcal{R} } \vert\widetilde{c}(\y)\vert^2
=
\prod_{j=1}^{d(n+1)} (R_j^{\frac{1}{2}} + Q^{\frac{m}{d}})^2 
\sum_{ \x \in \mathcal{O}_k^{n+1}\cap S_{F,\a}(T^{\frac{1}{d}}) } \vert c(\x)\vert^2.
\end{multline*}

The region $S_{F,\a}(T^{\frac{1}{d}})$ can be contained within the box 
$\{ (x_{i,\nu})_{0\leq i\leq n,\, \nu\mid \infty} \in\R^{d(n+1)}: \vert x_{i, \nu} \vert_\nu \leq T^{\frac{a_i}{d}} \}$ under the identification of  $\prod_{\nu \mid \infty} k_\nu^{n+1}$ with $\R^{d(n+1)}$. Therefore, we have $R_j \ll_k T^{\frac{a_i}{d}}$ for $id < j \leq (i+1)d$.
Concluding, we have
\begin{align*}
\Delta &\ll_k \prod_{j=1}^{d(n+1)} (R_j^{\frac{1}{2}} + Q^{\frac{m}{d}})^2
 \ll_k \prod_{i=0}^n \prod_{\lambda=1}^d (T^{\frac{a_i}{2d}} + Q^{\frac{m}{d}})^2
 \ll_k \prod_{i=0}^n (T^{a_i} + Q^{2m}).
\end{align*} 
Finally,  applying this with $T = cB$ for some $c$ depending on the field completes the proof.
\end{proof}


\section{Thin sets in weighted projective stacks} \label{sec:WPS}
In this section we prove Theorem~\ref{thm:stackyserre} using Theorem~\ref{thm:lopsided}. One delicate point is that $\P(\a)(\mathcal{O}_k) \neq \P(\a)(k)$ whenever $\a$ is non-trivial; we elucidate this difference in Corollary \ref{cor:IHP} below. Hence there is no natural reduction modulo $\mathfrak{p}$ map $\P(\a)(k) \to \P(\a)(\mathcal{O}_k/\mathfrak{p})$, which can make sieving modulo primes problematic. We avoid these issues by instead lifting to affine space: The map $(\A^{n+1}\setminus\{0\})(\mathcal{O}_k) \to \P(\a)(k)$ is not surjective in general. Nonetheless, there is a well-defined map $(\A^{n+1}(\mathcal{O}_k)) \setminus \{0\} \to \P(\a)(k)$ which is surjective. Let $\mathcal{A} = \{\mathfrak{a}_1, \dots, \mathfrak{a}_{h_k}\}$ be a system of representatives for the ideal classes in $k$.
Above any point $x \in \P(\a)(k)$, we may choose a representative tuple $\x \in \mathcal{O}_k^{n+1}$ such that $\wgcd_\a(\x) = \mathfrak{a}_i$ for some $i$.  This representative is unique modulo the action of the units.

Let $f: \sY \to \P(\a)$ be a thin morphism with $\sY$ integral. We choose a model for $\mathcal{Y}$ over $\mathcal{O}_k$, which by abuse of notation we also denote by $\mathcal{Y}$. Consider
$$
\xymatrix{
Y \ar[r] \ar[d]^g & \sY \ar[d]^f \\
\A^{n+1}\setminus\{0\} \ar[r] & \P(\a)
}
 $$
given by taking the fibre product with \eqref{def:WPS} . As $f$ is representable $g$ is representable, thus $Y$ is a scheme. 
Therefore we deduce that
\begin{multline*}
\#\{ x \in \P(\a)(k): H_\a(x) \leq B,\ x \in f(\sY(k))\}\\ 
\leq \#\{ \x \in [\mathcal{O}_k^{n+1}/ \mathcal{O}_k^\times] : \wgcd_\a(\x) \in \mathcal{A},\ H_\a(\x) \leq B,\ 
\x \in g(Y(k)) \}.\end{multline*}
By Nagata's compactification theorem we may extend $g$ to a proper morphism  $g^{\mathrm{c}}: Y^{\mathrm{c}} \to \A^{n+1}$. As the map $g^{\mathrm{c}}$ is generically finite, we may replace $g^{\mathrm{c}}$ by its relative normalistion, which will only change the counting problem by a negligible amount, to assume that $g^{\mathrm{c}}$ is finite. Then if $\x \in \mathcal{O}_k^{n+1}$ satisfies $\x \in g^{\mathrm{c}}(Y^{\mathrm{c}}(k))$, then actually  $\x \in g^{\mathrm{c}}(Y^{\mathrm{c}}(\mathcal{O}_k))$, since finite maps are integral  \cite[Tag~01WJ]{Stacks}.

Therefore we have reduced to obtaining upper bounds for the counting problem
\begin{equation} \label{eqn:upper_bound}
\#\{ \x \in \mathcal{O}_k^{n+1}/\mathcal{O}_k^\times : \wgcd_\a(\x) \in \mathcal{A},\ H_\a(\x) \leq B,\ \x \in \psi(Z(\mathcal{O}_k)) \}
\end{equation}
where $\psi: Z \to \A^{n+1}_{\mathcal{O}_k}$ is any non-birational finite morphism and $Z$ is an integral scheme. The quantity in \eqref{eqn:upper_bound} is then bounded above by
\begin{equation} 
\#\left\{  \x \in \mathcal{O}_k^{n+1}/\mathcal{O}_k^\times: \begin{array}{l}
H_{\a,\infty}(\x) \ll_k B ,\\ \x \bmod \mathfrak{p} \in \psi(Z(\mathcal{O}_k/\mathfrak{p})) \text{ for all } N_{k/\Q}(\mathfrak{p}) \leq Q\end{array}\right\}
\end{equation}
for any $Q \geq 1$. In the above, we have used the fact that there were only finitely many (with the precise number depending on $k$) choices for the weighted gcd, thus to find an upper bound we may drop this condition. 

We have now arrived at a point where we can employ the large sieve from Theorem~\ref{thm:lopsided}. As input, we need to understand the size of the sets modulo $\mathfrak{p}$ which we are excluding. We may do so using the following result of Serre.

%
%

\begin{lemma}[{\cite[\S 13.2 Thm.~5]{MW}}]\label{thm:modp}
There is a finite Galois extension $L_\psi/ k$ and a constant $0<c_\psi < 1$ such that if $\mathfrak{p} \subset \mathcal{O}_k$ splits completely in $L_\psi$ then
\[
\left \vert  \psi(Z(\mathcal{O}_k/\mathfrak{p})) \right \vert \leq c_\psi N_{k/\Q}(\mathfrak{p})^{n+1} + O\left( N_{k/\Q}(\mathfrak{p})^{n+ \frac{1}{2}} \right).
\]
\end{lemma}

We conclude the proof of the main theorem.

\begin{proof}[Proof of Theorem~\ref{thm:stackyserre}]
Let $Q \geq 1$ be a parameter which we will specify later on. Applying Theorem \ref{thm:lopsided} with $m=1$, we have
\begin{multline*}
\#\left\{  \x \in [\mathcal{O}_k^{n+1}/\mathcal{O}_k^\times]: \begin{array}{l}
H_{\a,\infty}(\x) \ll_k B ,\\ \x \bmod \mathfrak{p} \in \psi(Z(\mathcal{O}_k/\mathfrak{p})) \text{ for all } N_{k/\Q}(\mathfrak{p}) \leq Q\end{array}\right\}\\
\ll  \frac{\prod_{i=0}^n \left(B^{a_i} + Q^2\right)}{G(Q)}.
\end{multline*}
The conclusion of Lemma \ref{thm:modp} tells us that there is a constant $c \in (0,1)$ and a set of prime ideals $\mathcal{P}$ of positive density $\delta$ such that 
\begin{align*}
G(Q) &\geq 
\sum_{\substack{ N_{k/\Q}(\mathfrak{q}) \leq Q \\ \mathfrak{p} \subset \mathcal{O}_k,\ \mathfrak p \mid \mathfrak q \implies \mathfrak{p} \in \mathcal P}} \mu_k^2(\mathfrak{q}) \prod_{\mathfrak p \mid \mathfrak q} \frac{N_{k/\Q}(\mathfrak{p})^{n+1} - \#\psi(Z(\mathcal{O}_k/\mathfrak{p}))}{\#\psi(Z(\mathcal{O}_k/\mathfrak{p}))}\\
&\geq \sum_{\substack{ N_{k/\Q}(\mathfrak{q}) \leq Q \\ \mathfrak{p} \subset \mathcal{O}_k,\ \mathfrak p \mid \mathfrak q \implies \mathfrak{p} \in \mathcal P}} \mu_k^2(\mathfrak{q}) \prod_{\mathfrak p \mid \mathfrak q}  \left[ \frac{1-c}{c} + O(N_{k/\Q}(\mathfrak p)^{-\frac{1}{2}})\right].
\end{align*}
Let $g(\mathfrak q) = \mu_k^2(\mathfrak{q}) \prod_{\mathfrak p \mid \mathfrak q}  \left[ \frac{1-c}{c} + O(N_{k/\Q}(\mathfrak p)^{-\frac{1}{2}})\right]$ for $\mathfrak{q}$ only divisible by prime ideals in $\mathcal P$, and let $g(\mathfrak q)$ be 0 otherwise. Then, by the Chebotarev density theorem, we have
\begin{align*}
\sum_{N_{k/\Q}(p) \leq x} \frac{g(\mathfrak{p}) \log N_{k/\Q}(\mathfrak{p})}{N_{k/\Q}(\mathfrak{p})} &= \sum_{\substack{ N_{k/\Q}(\mathfrak{p}) \leq x \\ \mathfrak{p} \in \mathcal P}}  \frac{(1-c)\log N_{k/\Q}(\mathfrak{p})}{cN_{k/\Q}(\mathfrak{p})} + O(N_{k/\Q}(\mathfrak{p})^{-\frac{3}{2} + \epsilon})\\
& \sim \frac{\delta(1-c)}{c}  \log x.
\end{align*}
Hence an application of Wirsing's theorem (see e.g.~\cite[Lem.~3.11]{BLsieving} for a number field variant) yields
\[
G(Q) = \sum_{N_{k/\Q}(\mathfrak{q})\leq Q} g(\mathfrak{q}) \gg Q(\log Q)^{\eta \delta -1},
\]
for $\eta = \frac{1-c}{c}>0$. Choosing $Q = B^{\frac{\min_i a_i}{2 }}$ then gives the desired bound with $\gamma = 1 - \eta\delta$.
\end{proof}


We also give the following variant for integral points. It is a slightly stronger result which is included to highlight the difference between rational and integral points on proper stacks in general.

\begin{corollary} \label{cor:IHP}
	Let $\a \in \N^{n+1}$, let $k/\Q$ be a number field of degree $d$ and $\P(\a)$ the corresponding weighted projective stack, viewed over $\O_k$. Then $\P(\a)(\O_k)$ is not thin.
\end{corollary}
\begin{proof}
	 Every element of $\P(\a)(\O_k)$ admits a representative $(x_0,\dots,x_n) \in \O_k$ such that $\gcd(\x)$ has a fixed representative of an element of the class group (as opposed to $\wgcd(\x)$).
	 A minor variant of the argument given in \cite[Thm.~3.15]{BruinWPS} shows that the total number of integral points on $\P(\mathbf a)$ of height at most $B$ grows like $B^{a_0+\dots+a_n}$. Theorem \ref{thm:stackyserre} then immediately show that $\P(\a)(\O_k)$ is not thin.
\end{proof}

\section{Level structures on hyperelliptic curves} \label{sec:hyperelliptic}
In this section we 	prove Theorem \ref{thm:level_structures}. 

\subsection{Moduli stacks of abelian varieties}
We first introduce moduli stacks of level structures. For relevant background see \cite[\S IV.6]{FC90}, though we use slightly different conventions to \emph{loc.~cit} since we allow a $\GSp_{2g}(\Z/n\Z)$-action rather than just a $\Sp_{2g}(\Z/n\Z)$-action; this is to get a good moduli stack defined over $\Q$ (see \cite[Rem.~IV.6.12]{FC90}).

Let $A$ be a principally polarised abelian variety (ppav) over a field $k$ of characteristic $0$. A level-$n$ structure on $A$ is a choice of isomorphism $A[n] \cong (\Z/n\Z)^{2g}$ of group schemes which preserves the Weil pairing up to a fixed scalar.  Let $\mathcal{A}_g$ denote the moduli stack of principally polarised abelian varieties (ppav) over $\Q$ and $\mathcal{A}_{g,n}$ the moduli stack of ppav with level $n$-structure. There is a natural map $\mathcal{A}_{g,n} \to \mathcal{A}_g$ which is a $\GSp_{2g}(\Z/n\Z)$-torsor. The stack $\mathcal{A}_{g,n}$ is smooth but need not be connected. Let $G \subseteq \GSp_{2g}(\Z/n\Z)$ be a subgroup. We call the quotient $\mathcal{A}_{g,G}:= [\mathcal{A}_{g,n}/G]$ \textit{the stack of ppav with level-$G$ structure}. There is again a natural map $\mathcal{A}_{g,G} \to \mathcal{A}_g$ which is finite \'etale of degree $|\GSp_{2g}(\Z/n\Z)|/|G|$. We first determine the connected components. We use the short exact sequence
$$ 1 \to \Sp_{2g} \to \GSp_{2g} \to \Gm \to 1$$
where the last map is the multiplier map $\nu$. This sequence is split.

\begin{lemma} \label{lem:A_g_connected}
	Let $G \subseteq \GSp_{2g}(\Z/n\Z)$ be such that $\nu(G) = (\Z/n\Z)^\times$. Then $\mathcal{A}_{g,G}$ is geometrically connected over $\Q$. 
\end{lemma}
\begin{proof}
	The map $\mathcal{A}_{g,n} \to \Spec \Q$ factors as
	$\mathcal{A}_{g,n} \to \Spec \Q(\mu_n) \to \Spec \Q$ 
	and $\mathcal{A}_{g,n}$ is geometrically connected over $\Q(\mu_n)$ \cite[Cor.~IV.6.8]{FC90}.
	It follows that the Galois group of the geometric components of $\mathcal{A}_{g,n}$ over $\Q$ is 
	exactly $\Gal(\Q(\mu_n)/\Q)$, with different components corresponding to Galois orbits of 
	different choices of multiplier. Thus $\nu(\GSp_{2g}(\Z/n\Z)) = (\Z/n\Z)^\times$ acts on the 
	components freely and transitively via the Galois action. Hence taking a quotient by $G$ with $\nu(G) = (\Z/n\Z)^\times$ identifies the connected components and the result is clear.
\end{proof}

We next verify that this geometric construction agrees with the notion of $G$-level structures, in terms of images of Galois representations, used in \S\ref{sec:intro_hyperelliptic}.

\begin{lemma} \label{lem:level_structure_quotient_stack}
	Let $A \in \mathcal{A}_g(k)$ and $G \subseteq \GSp_{2g}(\Z/n\Z)$. Then $A$ lies in the image of the map 
	$\mathcal{A}_{g,G}(k) \to \mathcal{A}_{g}(k)$ if and only if the image of the  Galois representation $\rho_{A,n}: \Gamma_k \to \GSp_{2g}(\Z/n\Z)$ is contained in a conjugate of $G$.
\end{lemma}
\begin{proof}
	Taking successive quotients yields the commutative diagram of Cartesian squares
	\[
\xymatrix{
\mathcal{A}_{g,n} \ar[r] \ar[d] & \Spec k \ar[d]  \\
\mathcal{A}_{g,G} \ar[r] \ar[d] & BG \ar[d]  \\
\mathcal{A}_{g} \ar[r] &  B\GSp_{2g}(\Z/n\Z) }
\]
where $BG = [\Spec k /G]$ denotes the classifying stack of $G$. The bottom map is exactly the classifying map determined by the $\GSp_{2g}(\Z/n\Z)$-torsor $\mathcal{A}_{g,n} \to \mathcal{A}_{g}$. However $BG(k) = \H^1(k,G) = \Hom(\Gamma_k,G)/\hspace{-4pt}\sim$ where the equivalence is up to conjugacy in $G$ (this latter equality follow from the explicit description of non-abelian Galois cohomology in terms of cocycles and coboundaries, using the fact that $\Gamma_k$ acts trivially on $G$; see \cite[\S5.1]{Ser02}). The bottom arrow is thus exactly the map which associates to $A \in \mathcal{A}_g(k)$ the induced $\GSp_{2g}(\Z/n\Z)$-torsor, which is exactly equivalent to the data of the Galois representation $\rho_{A,n}$. By the universal property of the fibre product, we find that $A$ lies in the image of $\mathcal{A}_{g,G}(k) \to \mathcal{A}_{g}(k)$ if and only if the Galois representation $\rho_{A,n}$ lies in the image of $BG(k) \to B\GSp_{2g}(\Z/n\Z)$, i.e.~if and only if it lies in the image of $\Hom(\Gamma_k,G)/\hspace{-3pt}\sim \to \Hom(\Gamma_k,\GSp_{2g}(\Z/n\Z))/\hspace{-3pt}\sim$. This is exactly equivalent to the image of $\rho_{A,n}$ lying in $G$ up to conjugacy.
\end{proof}

\subsection{Moduli stacks of hyperelliptic curves}
For $i \in \{1,2\}$, let $\mathcal{H}_{2g + i}$ denote the moduli stack of odd/even hyperelliptic curves respectively. We have $\mathcal{H}_{2g+1} \subseteq \P(4,6,\cdots, 4g + 2)$ as a dense open subset by \cite[Prop.~5.9]{HP23} or \cite[Prop.~4.2(1)]{Fed14}. The Jacobian yields a morphism of stacks $\mathcal{H}_{2g+1} \to \mathcal{A}_{g}$; pulling back gives moduli stacks $\mathcal{H}_{2g+1,G} \to \mathcal{H}_{2g+1}$ of hyperelliptic curves with level-$G$ structure.

\begin{lemma} \label{lem:irreducible}
	Let $n \in \N$ be odd. 
	For any subgroup $G \subseteq \GSp_{2g}(\Z/n\Z)$, the map $\mathcal{H}_{2g +1,G} \to \mathcal{H}_{2g +1}$ 
	is finite \'etale. The stack $\mathcal{H}_{2g +1,G}$ is geometrically connected provided $\nu(G) = (\Z/n\Z)^\times$.
\end{lemma}
\begin{proof}
	The first part follows as $\mathcal{A}_{g,G} \to \mathcal{A}_g$ is finite \'etale.
	For geometric connectivity, the key fact is that the geometric monodromy action on 
	$\mathcal{H}_{2g +2,n}$, which comes from viewing it as a stack over $\C$, 
	is still the full $\Sp_{2g}(\Z/n\Z)$
	provided $n$ is odd by \cite[Thm.~1(2)]{ACamp}. Moreover, fixing a Weierstrass
	point does not affect level-$n$ structure for $n$ odd, hence $\mathcal{H}_{2g+1,n}$
	also has the full monodromy and is hence geometrically connected when viewed over $\Q(\mu_n)$.
	The result now follows from a similar argument to the proof of Lemma \ref{lem:A_g_connected}.
\end{proof}

\subsection{Proof of Theorem \ref{thm:level_structures}}
Let $G \subseteq \GSp_{2g}(\Z/n\Z)$ be a proper subgroup with $\nu(G) = (\Z/n\Z)^\times$. Then the map $\mathcal{H}_{2g +1,G} \to \mathcal{H}_{2g +1}$ is finite \'etale of degree greater than $1$, and $\mathcal{H}_{2g +1,G}$ is connected by Lemma \ref{lem:irreducible}. Hence it is a thin map in the sense of Definition \ref{def:thin}. Therefore we may apply Theorem \ref{thm:stackyserre} and Lemma~\ref{lem:level_structure_quotient_stack} to obtain the result.
\qed

\end{document}